\documentclass[10pt, leqno]{article}
\newenvironment{proof}{\noindent {\bf Proof:}}{$\Box$ \vspace{2 ex}}

\usepackage{amsmath, amssymb, amscd, enumerate, hyperref}
\usepackage{latexsym, epsfig, color, url, relsize}
\usepackage[all]{xy}
\usepackage{tikz, verbatim, xcolor}

\newcommand{\Cl}{\textnormal{Cl}}
\newcommand{\mfq}{\mathfrak{q}}
\newcommand{\mfp}{\mathfrak{p}}

\newcommand{\mff}{\mathfrak{f}}
\newcommand{\mfa}{\mathfrak{a}}
\newcommand{\mfb}{\mathfrak{b}}
\newcommand{\calP}{\mathcal{P}}

\newcommand{\C}{\mathbb{C}}

\newcommand{\prim}{{\rm{prim}}}

\newcommand{\all}{{\rm{all}}}
\newcommand{\Z}{\mathbb{Z}}

\newcommand{\Q}{\mathbb{Q}}
\newcommand{\R}{\mathbb{R}}
\newcommand{\F}{\mathbb{F}}

\newcommand{\cO}{\mathcal{O}}
\newcommand{\FF}{\mathcal{F}}
\newcommand{\GG}{\mathcal{G}}

\newcommand{\beq}{\begin{equation}}
\newcommand{\eeq}{\end{equation}}
\newcommand{\calO}{\mathcal{O}}

\newcommand\Aut{\operatorname{Aut}}

\newcommand\Disc{\operatorname{Disc}}

\newcommand\Gal{\operatorname{Gal}}
\newcommand\Res{\operatorname{Res}}

\newcommand\GL{\operatorname{GL}}

\newcommand\iD{\mathrm{D}}
\newcommand\iF{\mathrm{F}}

\renewcommand\b\bullet
\renewcommand\c\times

\newcommand{\LL}{{\mathcal L}}
\newcommand{\CC}{{\mathcal C}}
\newcommand{\om}{\omega}

\definecolor{dgreen}{RGB}{0, 170, 0}

\newtheorem{proposition}{Proposition}
\newtheorem{theorem}[proposition]{Theorem}
\newtheorem{corollary}[proposition]{Corollary}

\newtheorem{lemma}[proposition]{Lemma}
\newtheorem{remark}[proposition]{Remark}

\title{On the asymptotics of cubic fields ordered by general invariants}
     
\author{Arul Shankar and Frank Thorne}
       
\begin{document}

\maketitle
\begin{abstract}
In this article, we introduce a class of invariants of cubic fields termed ``generalized discriminants''. We then obtain asymptotics for the families of cubic fields ordered by these invariants. In addition, we determine which of these families satisfy the Malle--Bhargava heuristic.
\end{abstract}

\section{Introduction}

A foundational result due to Davenport--Heilbronn \cite{MR491593} provides asymptotics for
the number of real and cubic fields, when these fields are ordered by
their discriminants. Specifically, the theorem states:
\begin{theorem}[Davenport--Heilbronn]\label{thDH}
Let $N^{\pm}_{\Disc}(X)$ denote the number of cubic fields $K$, up to
isomorphism, that satisfy $|\Disc(K)|<X$ and $\pm\Disc(K)>0$. Then
\begin{equation*}
\begin{array}{rcccl}
N^{+}_{\Disc}(X)&=&\displaystyle\frac{1}{12\zeta(3)}X&+&o(X);\\[.2in]
N^{-}_{\Disc}(X)&=&\displaystyle\frac{1}{4\zeta(3)}X&+&o(X).
\end{array}
\end{equation*}
\end{theorem}

The above theorem, its extensions, and the methods of their proofs, have had a host of applications. Among many other applications, they are used by Yang \cite{MR2713725} to verify the Katz--Sarnak heuristics \cite{MR1659828} for low-lying zeroes of Dedekind zeta
functions of cubic fields; by
Bhargava--Wood \cite{MR2373587}, Belabas--Fouvry \cite{MR2740719} and Wang \cite{MR4219215} to prove Malle's conjecture for various different Galois
groups; by Martin--Pollack \cite{MR3022705} and Cho--Kim \cite{MR4211852} to obtain
the average value of the smallest prime satisfying certain prescribed
splitting conditions; by AS--S\"odergren--Templier \cite{AST} to prove that the Dedekind
zeta functions of infinitely many $S_3$-cubic fields have negative central values.

Theorem \ref{thDH} has also been generalized in numerous ways: Belabas--Bhargava--Pomerance \cite{MR2641942} prove power saving error terms; Bhargava \cite{MR2183288,MR2745272}
determines the asymptotics of quartic and quintic fields, when
ordered by discriminant; Datskovsky--Wright \cite{DW3}, Taniguchi \cite{TTNFcubic}, and Bhargava--AS--Wang \cite{global1} count cubic extensions of number fields and function fields; Belabas--Fouvry \cite{MR2740719} count subfamilies of cubic fields satisfying congruence conditions on their discriminants; Terr \cite{TerrThesis} proves that the ``shapes'' of cubic rings and fields are equidistributed (see also work of Bhargava--Harron \cite{MR3518306}, who give a uniform proof that shapes of cubic, quartic, and quintic rings and fields are equidistributed); Taniguchi--FT \cite{MR3127806} and Bhargava--AS--Tsimerman \cite{MR3090184} compute secondary terms (of size $\asymp X^{5/6}$) for the asymptotics of $N^{\pm}_{\Disc}(X)$. 

In this paper, we consider generalizations along a different direction: namely, we determine asymptotics for families of
cubic fields ordered by invariants more general than the
discriminant. Let $C(K)$ be the radical of $|\Disc(K)|$. That is, we
have $C(K):=\prod_{p\mid\Disc(K)}p$. We then prove the following
result. 
\begin{theorem}\label{thradical}
Let $N^\pm_C(X)$ denote the number of cubic fields $K$, up to
isomorphism, that satisfy $C(K)<X$ and $\pm\Disc(K)>0$. Then
\begin{equation*}
\begin{array}{rcl}
  N_C^+(X)&=&
  \displaystyle\frac{33}{120}\prod_{p}\Bigl(1+\frac{2}{p}\Bigr)
  \Bigl(1-\frac{1}{p}\Bigr)^2X\log X
  +o(X\log X);\\[.2in]
  N_C^-(X)&=&\displaystyle\left( \frac{3}{10} + \frac{33}{40} \right) \prod_{p}\Bigl(1+\frac{2}{p}\Bigr)
  \Bigl(1-\frac{1}{p}\Bigr)^2X\log X
  +o(X\log X).
\end{array}
\end{equation*}
\end{theorem}
\noindent Note that we break up the main term in the asymptotics for $N^-_C(X)$ into two summands; they correspond to what can be considered two disjoint subfamilies of cubic fields, namely, the family of pure cubic fields and the family of non-pure cubic fields.

\medskip

Theorem \ref{thradical} will be deduced as a special case of a more
general result that counts cubic fields ordered by various different
types of invariants.

\subsection*{Generalized discriminants of cubic fields}


Let $M$ be a Galois sextic field with Galois group $S_3$ over $\Q$. Then $K$ has three cubic $S_3$-subfields, which are conjugate to each other. One would therefore expect to be able to understand the family of sextic $S_3$-fields via the
family of cubic $S_3$-fields. Bhargava--Wood \cite{MR2373587} and Belabas--Fouvry \cite{MR2740719} independently use this philosophy to prove the
following result.
\begin{theorem}[Belabas--Fouvry, Bhargava--Wood]\label{thm:s3}
Let $N^\pm_{\Delta_6}(X)$ denote the number of Galois sextic number
fields $M$ with Galois group $S_3$, such that $|\Disc(M)|<X$ and
$\pm\Disc(M)>0$.

Then, we have
\begin{equation*}
N^\pm_{\Delta_6}(X)=\frac{C^\pm}{12}\prod_p c_p\cdot X^{1/3}+o(X^{1/3}),
\end{equation*}
where $C^+=1,C^-=3$, the product is over all primes, and
\begin{equation*}
c_p=\begin{cases}
(1-p^{-1})(1+p^{-1}+p^{-4/3}) & p\neq3,\\
(1-\frac13)(\frac43+\frac1{3^{5/3}}+\frac2{3^{7/3}})
& p=3.
\end{cases}
\end{equation*}
\end{theorem}
A power saving error term for the above quantity was obtained by work of Taniguchi--FT \cite{MR3208868}. In this work, they also speculate about a possible secondary term, and discuss tensions between theoretical predictions and the data.

Similarly to $C(K)$, we will regard $|\Disc(M)|$ as a ``generalized
discriminant'' of its cubic subfield. More specifically, let $K$ be a non-Galois cubic field, and denote the Galois closure of $K$ by $M$. Then $M$ has a unique quadratic subfield, denoted $L$. We say that $L$ is the {\it quadratic resolvent field} of $K$. Denote the discriminant of the quadratic resolvent $L$ of $K$ by $\iD(K)$. Then
$\iD(K)\mid \Disc(K)$, and moreover, $\Disc(K)/\iD(K)$ is always a
perfect integer square. Denote its positive integer squareroot by
$\iF(K)$. We note that apart from a factor of a bounded power of $3$,
the quantity $\iF(K)$ is simply the product of primes that totally
ramify in $K$, where $p$ is said to totally ramify in $K$ if $p$ splits as $p=\mathfrak{p}^3$. Similarly, up to a bounded power of $2$, the quantity $\iD(K)$ is the product of of primes that ramify, but not totally, in $K$. For a cubic $S_3$-field $K$, let $\Delta_6(K)$ denote the
discriminant of the Galois closure $M$ of $K$.  Then
we have the decompositions
\begin{equation}\label{eq:disc_decomp}
\Disc(K) = \iD(K) \iF(K)^2,\quad \Delta_6(K)=\iD(K)^3\iF(K)^4,\quad
C(K)=|\iD(K)|\iF(K),  
\end{equation}
where the final equality is true up to bounded factors of $2$
and $3$. For positive real numbers $\alpha$ and $\beta$, we say that
the invariant $|\iD|^\alpha\iF^\beta$ is a {\it generalized
  discriminant}.  This notion of generalized discriminant encompasses
all three invariants we have seen so far, namely, $\Disc(K)$,
$\Delta_6(K)$, and $C(K)$.

When $K$ is a cyclic cubic field, the invariant $\Delta_6(K)$ has no special meaning
but an otherwise similar analysis holds with $\iD(K) := 1$. We also define the above
quantities analogously when $K$ is a cubic \'etale extension of $\Q_p$.

Let $\Sigma=(\Sigma_v)_v$ be a {\it collection of cubic
  splitting types}, where for each place $v$ of $\Q$, the set
$\Sigma_v$ is the set of cubic \'etale extensions
of $\Q_v$ with specified inertial and ramification indices.\footnote{This is a less general notion than the one which allows $\Sigma_v$ to be an arbitrary subset of \'etale cubic extensions of $\Q_v$. We restrict ourselves to this less general notion for two reasons. First, this is the more natural notion from the point of view of families of $L$-functions; see the discussion on Sato--Tate equidistribution at the end of the introduction. Second, we did not 
obtain a version of Theorem \ref{thm:cm_explicit_lc} valid in this generality. Although this
seems likely to be possible, it appears liable
to be inelegant while presenting additional complications in the proof.}
The collection $\Sigma$ is said to be a {\it finite
collection} if for all large enough primes $p$, $\Sigma_p$ is the
set of all cubic \'etale extensions of $\Q_p$ (i.e., all inertial and ramification indices are allowed).
Throughout, we write $P_\Sigma$ for the product of those primes
where $\Sigma_p$ is a proper subset of these extensions.

Given a finite
collection of cubic splitting types $\Sigma$, let $\FF(\Sigma)$
denote the set of cubic fields $K$ such that $K\otimes\Q_v\in\Sigma_v$
for all $v$. For a generalized discriminant $I$, we define
\begin{equation*}
N_{I}(\Sigma;X):=\#\{K\in\FF(\Sigma):I(K)<X,\, \iD(K) \neq -3\}.
\end{equation*}
As the pure cubic fields (those with $\iD(K) = -3$) behave differently from those with
other quadratic resolvents, we will treat them separately.

The next result determines asymptotics for the family
$\FF(\Sigma)$, excluding the pure cubic fields, ordered by generalized discriminants.

\begin{theorem}\label{thGD}
Fix positive real numbers $\alpha$ and $\beta$, and let
$I=|\iD|^\alpha\iF^\beta$ be a generalized discriminant. Let $\Sigma$ be
a finite collection of cubic splitting types. Then
\begin{itemize}
\item[{\rm (a)}] When $\alpha<\beta$, we have
\begin{equation*}
N_I(\Sigma;X) = \frac12\Bigl(\sum_{K\in\Sigma_\infty}\frac{1}{|\Aut(K)|}\Bigr)
\prod_p\Bigl(\sum_{K\in\Sigma_p}\frac{|\iD(K)|_p|\iF(K)|_p^{\beta/\alpha}}{|\Aut(K)|}\Bigr)
\Bigl(1-\frac1{p}\Bigr)
\cdot X^{\frac{1}{\alpha}}
+O_{\epsilon,I}\big( (X^{\frac{2}{\alpha+\beta}+\epsilon}+X^{\frac{5}{6\alpha}}) P_\Sigma^{2/3} \big).
\end{equation*}
\item[{\rm (b)}] When $\alpha>\beta$, we have
\begin{equation*}
N_I(\Sigma;X)= 
\Bigl(\sum_{d\,{\rm fund.}\,{\rm disc} \neq -3}\frac{\Res_{s=1}\Phi_{\Sigma,d}(s)}{|d|^{\alpha/\beta}}
\Bigr)\cdot X^{\frac{1}{\beta}}+O_{\epsilon,I}\big((X^{\frac{3}{2\alpha+\beta}+\epsilon}+X^{\frac{2}{3\beta}+\epsilon}) P_\Sigma^{1/3} \big),
\end{equation*}
where $\Phi_{\Sigma,d}(s)$ are Dirichlet series introduced in \S2.
\item[{\rm (c)}] When $\alpha=\beta$, we have
\begin{equation*}
N_I(\Sigma;X)=
\frac1{2\alpha} \Bigl(\sum_{K\in\Sigma_\infty}\frac{1}{|\Aut(K)|}\Bigr)
\prod_{p}\Bigl(\sum_{K\in\Sigma_p}
\frac{|\iD(K)|_p|\iF(K)|_p}{|\Aut(K)|}\Bigr)\Bigl(1-\frac1{p}\Bigr)^2
\cdot X^{\frac{1}{\alpha}}\log X+o_{\Sigma,I}(X^{\frac1{\alpha}}\log X).
\end{equation*}
\end{itemize}
\end{theorem}

For the pure cubic fields, Cohen and Morra
proved \cite[Corollary 7.4]{CM} that, when $\Sigma_v = \Sigma_v^{\all}$ for all $v$,
\begin{equation}\label{eq:CM_result1}
\#\{K\in\FF(\Sigma): \iD(K) = -3, \iF(K) < Z\}
=
C_1 Z (\log(Z) + C_2 - 1) + O(Z^{2/3 + \epsilon}),
\end{equation}
where
\[
C_1 := \frac{7}{30} \prod_p \Bigl(1+\frac{2}{p}\Bigr)
  \Bigl(1-\frac{1}{p}\Bigr)^2,
\]
\[
C_2 := 2\gamma - \frac{16}{35} \log(3) + 6 \sum_{p} \frac{\log(p)}{p^2 + p - 2},
\]
where the sum and product are over all primes $p$. This result also generalizes to arbitrary
$\Sigma$; see \eqref{eq:CM_result2}. Taking $Z = X^{1/\beta} 3^{-\alpha/\beta}$, we see
that adding the pure cubic fields adds a term of order $X^{1/\beta} \log(X)$, along with
a secondary term of order $X^{1/\beta}$,
to each of the results in Theorem \ref{thGD}. For (a) this is subsumed by the error term,
and the result is unchanged; for (b), this new contribution dominates the asymptotics by a factor of $\asymp\log X$, so that asymptotically 100\% of cubic fields ordered by $I$
will be pure cubic fields; for (c) this contribution is of equal magnitude, and the pure
and non-pure cubic fields each constitute a positive proportion of cubic fields ordered 
by $I$.

We recover Theorem \ref{thm:s3}, with a power saving error term of
$O(X^{\frac{2}{7} + \epsilon})$, by taking $\alpha = 3$ and $\beta =
4$ in Theorem \ref{thGD} and carrying out an appropriate calculation at the $2$- and
$3$-adic places.  (This was also noted in \cite{BTT}.) When
$\frac{\beta}{\alpha} > \frac{7}{5}$, the error term of
$O(X^{\frac{5}{6\alpha}})$ in case (a) dominates the other error term
and can be refined into a secondary term extrapolating that proved in
\cite{MR3090184,MR3127806} for $\alpha = 1$ and $\beta = 2$. More precisely, we have the following result.

\begin{theorem}\label{thm:secterm}
Let $\alpha$ and $\beta$ be positive real numbers with $\frac{\beta}{\alpha} > \frac{7}{5}$, and let $I=|\iD|^\alpha\iF^\beta$. Then we have
\begin{equation*}
N_I(\Sigma;X) = C_1(I;\Sigma)\cdot X^{\frac{1}{\alpha}}
+C_{2}(I;\Sigma)\cdot X^{\frac{5}{6\alpha}}
+O_\epsilon\big( (X^{\frac{2}{\alpha+\beta}+\epsilon}+X^{\frac{2}{3\alpha}+\epsilon}) P_\Sigma^{2/3} \big),
\end{equation*}
where $C_1(I;\Sigma)$ is the leading constant appearing in the right hand side of the displayed equation in Part (a) of the above theorem, and 
\begin{equation*}
C_{2}(I;\Sigma)=C(\infty) \frac{ 4 \zeta(1/3)}{5 \Gamma(2/3)^3 \zeta(5/3)}
\prod_p
\left[
\frac{\text{\footnotesize $\displaystyle\sum_{K_p\in\Sigma_p(f)}$} 
\!\frac{|\iD(K_p)|_p |\iF(K_p)|_p^{\frac{5\beta + 2\alpha}{6\alpha}}}{|\Aut(K_p)|} \text{\footnotesize $\displaystyle\int_{\scriptsize\calO_{K_p}\setminus p\calO_{K_p}}$}\text{\footnotesize $\!\!\![\calO_{K_p}\!:\!\Z_p[x]]^{2/3}dx$}}
{\text{\footnotesize $\displaystyle\sum_{K_p\in \Sigma_p^{\all}}$} 
\!\frac{|\iD(K_p)|_p |\iF(K_p)|_p^2}{|\Aut(K_p)|} \text{\footnotesize $\displaystyle\int_{\calO_{K_p}\setminus p\calO_{K_p}}$}\text{\footnotesize $\!\!\![\calO_{K_p}\!:\!\Z_p[x]]^{2/3}dx$}}
\right]
\end{equation*}
where $C(\infty)$ is $1$, $\sqrt{3}$ or $1+\sqrt{3}$
depending on whether $\Sigma_\infty$ consists of $\R^3$, $\R\oplus\C$ or both, respectively. Also, $\cO_K$ denotes the ring of integral elements in $K_p$.
\end{theorem}

\subsection*{The Malle--Bhargava heuristics}

In \cite{MR1884706,MR2068887}, Malle develops heuristics for asymptotics of the number of degree-$n$ number fields with Galois group $G$ and bounded discriminant, where $n>1$ is any integer and $G$ is a finite group with an action on a set with $n$ elements. These heuristics are believed to be true in most cases. However, see \cite{MR2135320} where Kl\"{u}ners demonstrates a counter example in the case $n=3$ and $G=C_3\wr C_2$, and \cite{MR3349443}, where T\"{u}rkelli modifies Malle's conjecture so that it holds in the above and similar cases. While Malle's conjecture has been formulated only for families of fields ordered by discriminant, the same method applies to other orderings, in particular to the generalized discriminants that we work with. 

Interestingly, the leading constants appearing in front of Malle's heuristics are still shrouded with mystery. In the case of degree-$n$ $S_n$ number fields ordered by discriminant, Bhargava \cite{MR2354798} formulates a conjecture for the leading coefficients, using a general recipe which constructs these constants from mass formulae counting \'etale extensions of local fields. Once again, this recipe is quite general, applying to any family of number fields 
constructed as follows: fix a degree $n>1$ and a group $G$ with a transitive action on the set $\{1,\ldots,n\}$. Then this recipe applies to the family of all degree-$n$ number fields with Galois group $G$, satisfying any finite set of splitting conditions, ordered by any generalized discriminant. (See also work of Kedlaya \cite{MR2354797} describing how these leading constants can be computed in the more general case of families of Galois representations.)
However, there are many instances where this prediction gives the incorrect leading constant. The prototypical example is the family of quartic $D_4$-fields ordered by discriminant, where the asymptotic constant determined by Cohen--Diaz y Diaz--Olivier \cite{MR1918290} is not expected to equal the constant that this recipe would predict. On the other hand, when quartic $D_4$-fields are ordered by conductor, Altu\u{g}--AS--Varma--Wilson \cite{D4preprint} establish that the leading asymptotic constant does arise from the Malle--Bhargava recipe. This leads to the natural question, as
discussed by Bhargava in \cite{MR2354798}, of which families of number fields ordered by which invariants satisfy this property.

We say that a family $F$ of number fields, ordered by some generalized discriminant, satisfies the {\it Malle--Bhargava heuristic} if the asymptotics of every subfamily defined by prescribed splitting at finitely many primes are as predicted by the Malle--Bhargava recipe. (Despite our terminology,
we emphasize again that Bhargava conjectured this only for $S_n$, and did not predict that it should always hold.)

A necessary condition is that the splitting behaviour of primes is independent. We now precisely define this notion. Let $\GG$ be a family of number fields having the same degree
$n$.\footnote{It is not entirely clear exactly what constitutes a {\it
    family} of number fields. Being the set of all number fields
  having the same degree and the same Galois group is assumed to be a
  sufficient though not a necessary condition. See \cite{MR3675175}, where a similar question is discussed in detail.} Let $\Sigma=(\Sigma_v)_v$ be a {\it
  collection of degree-$n$ splitting types}, where for each place $v$ of $\Q$,
$\Sigma_v$ is the set of degree-$n$ \'etale extensions of $\Q_v$ satisfying specified inertial and ramification behaviour. 
For
each place $v$, let $\Sigma_v^\all$ denote the set of all degree-$n$
\'etale extensions of $\Q_v$. Then $\Sigma$ is said to be {\it finite}
if $\Sigma_p=\Sigma_p^\all$ for all sufficiently large primes $p$.
Let $h:\GG\to\R_{>0}$ be a height function (i.e., there are only
finitely many elements of $\GG$ having bounded height). Let
$N_{h}(\GG_\Sigma;X)$ denote the number of elements in $\GG$ satisfying $\Sigma$ and
having height less than $X$. Then we say that the family $\GG$ ordered
by $h$ satisfies {\it independence of primes} if the following is
true. For all places $v$ of $\Q$, there exist functions
$\sigma_v:\Sigma_v^\all\to\R_{\geq 0}$ with
\begin{equation*}
\sum_{K_v\in\Sigma_v^\all}\sigma_v(K_v)=1,
\end{equation*}
such that the following condition is satisfied. For each finite
collection of splitting types $\Sigma$, we have
\begin{equation*}
N_{h}(\GG_\Sigma;X)\sim \Bigl(\prod_{v}\sum_{K_v\in\Sigma_v}
\sigma_v(K_v)\Bigr)\cdot N_{h}(\GG;X).
\end{equation*}
There are many known examples of families of number fields which do not satisfy independence of primes. See for example \cite{MR2581243}, in which Wood studies families of number fields with any fixed abelian Galois group, and proves in many cases that, when ordered by discriminant, these families do not satisfy independence of primes. We note that the notion of satisfying independence of primes is a weaker notion than that of satisfying the Malle--Bhargava heuristic, when both these notions make sense. Moreover, independence of primes can be defined for a wider class of families, for example, this notion makes sense for the family of pure cubic fields, the family of monogenic degree-$n$ fields, and many other families for which the Malle--Bhargava heuristics do not apply.

Next, we consider the family of all cubic fields. It is natural to partition this family into two subfamilies: the family of pure cubics and the family of non-pure cubics. The ordering on the family of pure cubic fields coming from any generalized discriminant is the same (since we have $\iD(K)=-3$ for every pure cubic field $K$). It follows from the method of Cohen--Morra \cite{CM} described in \S2.2 that the family of pure cubic fields satisfies independence of primes. 
For the family of non-pure cubic fields ordered by generalized
discriminants, we have the following result.
\begin{theorem}\label{thIndPr}
Let $I=|\iD|^\alpha\iF^\beta$ be a generalized discriminant. Then
the family of all non-pure cubic fields
ordered by $I$ satisfies independence of primes and the Malle--Bhargava heuristic if and only if $\alpha \leq \beta$.
\end{theorem}

For the $\alpha\leq\beta$ case, the above result is an immediate consequence of Theorem \ref{thGD}. This $\alpha>\beta$ case requires a bit more
work, since the residues of the Dirichlet series appearing in Part (b) of Theorem \ref{thGD} are not explicit. We give a general proof which also applies to many different situations, such as the family of quartic $D_4$-fields ordered by discriminant. 


\medskip

Finally, our counting results also have implications towards families of Artin $L$-functions associated to cubic $S_3$-fields. Indeed, let $\rho:S_3\to\GL_n(\C)$ be any representation of $S_3$. Given a cubic $S_3$-field $K$, with normal closure $M$, we obtain a Galois representation
\begin{equation*}
\Gal(\overline{\Q}/\Q)\to\Gal(M/\Q)\cong S_3\to\GL_n(\C),
\end{equation*}
where the final map is $\rho$. We associate to this Galois representation its Artin $L$-function, denoted $L(s;\rho,K)$. Throughout, we assume that $\rho$ contains at least one copy of the standard representation of $S_3$, which is necessary to ensure that different cubic fields give rise to different $L$-functions. Then, given a family $\FF(\Sigma)$ of cubic $S_3$-fields $K$, we obtain a family of Artin $L$-functions $L(s;\rho,K)$ that we denote by $\LL(\rho,\Sigma)$. We order the $L$-functions in $\LL(\rho,\Sigma)$ by their conductors.

Ordering $\LL(\rho,\Sigma)$ by conductor corresponds to ordering $\FF(\Sigma)$ by a certain generalized discriminant $I=|\iD|^\alpha\iF^\beta$ depending on $\rho$. Indeed, we have 
$(\alpha, \beta) = c_1(1, 2) + c_2(1, 0)$ where $c_1 \geq 1$ and $c_2 \geq 0$ are the multiplicities of
the standard and sign representations respectively, so that $\alpha > 0$ and $\beta > 0$.
A consequence of Theorem \ref{thGD} is that the family $\LL(\rho,\Sigma)$ satisfies Sato--Tate equidistribution in the sence of \cite[Conjecture 1]{MR3675175}. Loosely speaking, a family of $L$-functions arising from number fields satisfies Sato--Tate equidistribution when the asymptotics of these number fields, ordered by the conductors of their $L$-functions, satisfy the Malle-Bhargava heuristics on average over primes $p$. Identically to the arguments in \cite[\S3.1]{MR3993807}, when $\alpha\leq\beta$, this follows immediately from the shape of the leading constant in Parts (a) and (c) of Theorem \ref{thGD}. When $\alpha>\beta$ the situation is similar to the case of the family of Dedekind zeta functions of $D_4$-fields considered in \cite[\S6.2]{MR3993807}. As there, we consider the family of cubic fields ordered by $I$ to be a countable union of subfamilies, one for each fixed quadratic resolvent field. Since each of these subfamilies contributes a positive proportion to the full family, Sato--Tate equidistribution for the full family follows from Sato--Tate equidistribution for each subfamily. Thus, we have the following consequence to Theorem \ref{thGD}:
\begin{corollary}
With notation as above, the families $\LL(\rho,\Sigma)$ satisfy Sato--Tate equidistribution.
\end{corollary}
It is interesting to note that despite independence of primes not always holding, Sato--Tate equidistribution is always satisfied for our families.

\subsection*{Organization of the paper}
We begin in \S2 by considering families of cubic fields with one fixed invariant. Invoking work of Bhargava--Taniguchi--FT \cite{BTT} on the Davenport-Heilbronn theorem, we obtain asymptotics for families of cubic fields with fixed $\iF$; using work of Cohen--Morra \cite{CM} and Cohen--FT \cite{CT} on a Kummer-theoretic approach, we deduce asymptotics for families of cubic fields with fixed $\iD$. The leading constants appearing in the asymptotics for the latter family are somewhat inexplicit, but in \S3 we prove that the average values of these constants have an explicit description given in terms of products of mass formulae.

The results of the previous two subsections allow us to determine asymptotics for $\FF(\Sigma)$ ordered by generalized discriminants. This is accomplished in \S4, and we extract secondary terms and power saving error terms when possible. We then establish exactly when independence of primes holds, thereby proving Theorem \ref{thIndPr}. Finally, we conclude in \S5 by presenting some numerical data.

Throughout, implied constants may
depend on $\epsilon$, $\alpha$, and $\beta$, but not $\Sigma$ unless otherwise noted.

\section{Families of cubic fields with a fixed invariant}

Recall that for each cubic field or \'etale algebra $K/\Q$
or $K/\Q_p$, we have a decomposition
\begin{equation}\label{eq:disc_decomp_repeat}
\Disc(K) = \iD(K)\iF(K)^2,
\end{equation}
where $\iD(K)$ is the discriminant of the {\itshape quadratic
  resolvent} algebra of $K$. When $K$ is a $S_3$-cubic field 
  $\iD(K)$ is the discriminant of the unique quadratic field contained in the
Galois closure of $K$, and when $K$ is a cyclic cubic field $\iD(K) = 1$.
We decompose these quantities into local factors
\begin{equation}
  \Disc(K) = \pm \prod_p \Disc_p(K),\quad \iD(K) = \pm \prod_{p} \iD_p(K),
  \quad \iF(K) = \prod_p \iF_p(K),
\end{equation}
with $\iD_p(K) = p^{v_p(\iD(K))}$ and $\iF_p(K) = p^{v_p(\iF(K))}$.
Then these quantities enjoy the following properties:
\begin{itemize}
\item[{\rm (a)}] When $p > 3$, then 
\[
(\iD_p(K), \iF_p(K)) 
\in \{ (1, 1), (p, 1), (1, p) \},
\]
with the three cases corresponding to 
the ramification type of $p$ in $K$: 
unramified, partially ramified, or totally ramified,
respectively.
\item[{\rm (b)}] When $p = 3$, we have
\[
(\iD_3(K), \iF_3(K)) \in \{ (1, 1), (p, 1), (p, p), (1, p^2), (p, p^2) \}.
\]
Here $p$ is unramified in the first case, partially ramified in the second case,
and totally ramified in the remaining cases.
\item[{\rm (c)}] When $p = 2$, we have
\[
(\iD_2(K), \iF_2(K)) \in \{ (1, 1), (p^2, 1), (p^3, 1), (1, p) \}.
\]
Here $p$ is unramified in the first case, partially ramified in the
next two cases, and totally ramified in the last case.
\end{itemize}

Given a positive number $f$, squarefree away from $3$, and indivisible by $27$, we let $\FF(\Sigma)^{(f)}$ denote the set of cubic $S_3$-fields $K\in\FF(\Sigma)$ with $\iF(K)=f$. Given a fundamental discriminant $d$, we let $\FF(\Sigma)_d$ denote the set of cubic $S_3$-fields $K\in\FF(\Sigma)$ with $\iD(K)=d$. (By convention, we consider $1$ to be a fundamental discriminant.)
In this section, we obtain asymptotics for the number of $K\in\FF(\Sigma)^{(f)}$ with $|\iD(K)|<Y$ in \S\ref{subsec:fixedF}, and the number of $K\in\FF(\Sigma)_{d}$ with $\iF(K)<Z$ in \S\ref{sec:small_d}. In particular, we obtain error terms that control the dependence on $\Sigma$.

\subsection{Counting cubic fields $K$ with fixed $\iF(K)$}\label{subsec:fixedF}

Let $f$ be a fixed positive integer, squarefree away from $3$. To count cubic
fields $K$ where $\iF(K)=f$, we appeal to a strengthening of the
Davenport-Heilbronn theorem.
Define the quantity
\[
N(\FF(\Sigma)^{(f)};Y) :=  \#\{ K\in\FF(\Sigma)^{(f)} :  |\iD(K)| < Y\}.
\]
Then we have the following:
\begin{theorem}[\cite{BTT}, Theorem 1.4]\label{thm:dh}
We have
\beq\label{eq:dh}
N(\FF(\Sigma)^{(f)};Y) =  C_1(\Sigma, f) \cdot Y 
+ C_2(\Sigma,f)\cdot Y^{5/6} + O\big(E(Y; f,\Sigma)\big),
\eeq
for constants $C_1(\Sigma,f)$ and $C_2(\Sigma,f)$ described below, and with the following
`averaged' bound on $E(Y; f, \Sigma)$: for each $f \leq F$, 
choose independent and arbitrary values $Y_f \leq Y$. Then, we
have
\[
\sum_{f \leq F} E(Y_f; f, \Sigma) \ll_\epsilon Y^{2/3 + \epsilon} F^{4/3+\epsilon} P_\Sigma^{2/3},
\]
uniformly in $F$.
\end{theorem}
The leading constant $C_1(\Sigma, f)$ is described as follows. First, for a
prime $p$ and a positive integer $f$, define the set $\Sigma_p(f)$ of
{\it $f$-compatible algebras in $\Sigma_p$} to be those \'etale cubic
extensions $K_p$ of $\Q_p$ such that the powers of $p$ dividing
$\iF(K_p)$ and $f$ are the same. Then we have
\[
C_1(\Sigma, f) := \frac12\Bigl(\sum_{K_\infty\in\Sigma_\infty}
\frac{1}{|\Aut(K_\infty)|}\Bigr)
\prod_{p}\Bigl[\Bigl(\sum_{\substack{K_p\in\Sigma_p(f)}}
\frac{|\iD(K_p)|_p}{|\Aut(K_p)|}\Bigr)\Bigl(1-\frac1{p}\Bigr)\Bigr].
\]
For each prime $p>3$, when $\Sigma_p=\Sigma_p^\all$, we have
\begin{equation}\label{eq:shapeC1}
\Bigl(\sum_{\substack{K_p\in\Sigma_p(f)}}
\frac{|\iD(K_p)|_p}{|\Aut(K_p)|}\Bigr)\Bigl(1-\frac1{p}\Bigr)=
\left\{
\begin{array}{rcl}
\Bigl(1-\frac{1}{p}\Bigr)&\mbox{when}& p\mid f;\\[.2in]
\Bigl(1-\frac{1}{p^2}\Bigr)&\mbox{when}& p\nmid f.
\end{array}\right.
\end{equation}
Meanwhile, the secondary constant $C_2(\Sigma,f)$ is given by
\[
C_2(\Sigma, f) := C(\infty) \frac{ 4 \zeta(1/3)}{5 \Gamma(2/3)^3 \zeta(5/3)}
\prod_p \nu_p(\Sigma_p, f),
\]
where $C(\infty)$ is $1$, $\sqrt{3}$, or $1+\sqrt{3}$ depending on whether $\Sigma_\infty$ consists of $\R^3$, $\R\oplus\C$, or both, respectively, and 
\begin{equation}\
\nu_p(\Sigma_p, f) := \frac{\text{\footnotesize $\displaystyle\sum_{K_p\in\Sigma_p(f)}$} 
\!\frac{|\iD(K_p)|_p |\iF(K_p)|_p^{1/3}}{|\Aut(K_p)|} \text{\footnotesize $\displaystyle\int_{\scriptsize\calO_{K_p}\setminus p\calO_{K_p}}$}\text{\footnotesize $\!\!\![\calO_{K_p}\!:\!\Z_p[x]]^{2/3}dx$}}
{\text{\footnotesize $\displaystyle\sum_{K_p\in \Sigma_p^{\all}}$} 
\!\frac{|\iD(K_p)|_p |\iF(K_p)|_p^2}{|\Aut(K_p)|} \text{\footnotesize $\displaystyle\int_{\calO_{K_p}\setminus p\calO_{K_p}}$}\text{\footnotesize $\!\!\![\calO_{K_p}\!:\!\Z_p[x]]^{2/3}dx$}}.
\end{equation}
Moreover we have $C_1(\Sigma, f) < 1$ and $|C_2(\Sigma, f)| \ll f^{-1/3}$ for all $\Sigma$ and~$f$.

To compute average values of these constants, we introduce the Dirichlet series $L_1(\Sigma, s)$ and $L_2(\Sigma, s)$ given by
\begin{equation}\label{eq:constantL}
L_1(\Sigma, s):=\sum_{f}C_1(\Sigma,f)f^{-s};
\quad
L_2(\Sigma, s):=\sum_{f}C_2(\Sigma,f)f^{-s};
\end{equation}
These series satisfy the following Euler product decomposition in their domains of absolute convergence.
\begin{proposition}\label{prop:L12Euler}
For $\Re(s)>1$, we have
\begin{equation*}
\begin{array}{rcl}
L_1(\Sigma, s)&=&\displaystyle
\frac12\Bigl(\sum_{K\in\Sigma_\infty}\frac{1}{|\Aut(K)|}\Bigr)\prod_{p}\Bigl(\sum_{K\in\Sigma_p}
\frac{|\iD(K)|_p|\iF(K)|_p^s}{|\Aut(K)|}\Bigr)\Bigl(1-\frac1{p}\Bigr);
\\[.2in]
L_2(\Sigma, s-1/3)&=&\displaystyle
C(\infty) \frac{ 4 \zeta(1/3)}{5 \Gamma(2/3)^3 \zeta(5/3)}
\prod_p
\left[
\frac{\text{\footnotesize $\displaystyle\sum_{K_p\in\Sigma_p}$} 
\!\frac{|\iD(K_p)|_p |\iF(K_p)|_p^{s}}{|\Aut(K_p)|} \text{\footnotesize $\displaystyle\int_{\scriptsize\calO_{K_p}\setminus p\calO_{K_p}}$}\text{\footnotesize $\!\!\![\calO_{K_p}\!:\!\Z_p[x]]^{2/3}dx$}}
{\text{\footnotesize $\displaystyle\sum_{K_p\in \Sigma_p^{\all}}$} 
\!\frac{|\iD(K_p)|_p |\iF(K_p)|_p^2}{|\Aut(K_p)|} \text{\footnotesize $\displaystyle\int_{\calO_{K_p}\setminus p\calO_{K_p}}$}\text{\footnotesize $\!\!\![\calO_{K_p}\!:\!\Z_p[x]]^{2/3}dx$}}
\right]
.
\end{array}
\end{equation*}
\end{proposition}
\begin{proof}
To prove the first equality in the above displayed notation, note that we have 
\begin{equation} \label{eqLs}
\begin{array}{rcl}
\displaystyle L_1(\Sigma, s)
&=&\displaystyle\sum_{f\geq 1}\frac{C_1(\Sigma,f)}{f^{s}}
\\[.2in]&=&
\displaystyle
\frac12 \Big(\sum_{K\in\Sigma_\infty}\frac{1}{|\Aut(K)|}\Big)\sum_{f\geq 1}\frac{1}{f^s}
\prod_p\Bigl(\sum_{K\in\Sigma_p(f)}
\frac{|\iD(K)|_p}{|\Aut(K)|}\Bigr)\Bigl(1-\frac1{p}\Bigr)
\\[.2in]&=&\displaystyle
\frac12\Big(\sum_{K\in\Sigma_\infty}\frac{1}{|\Aut(K)|}\Big)\prod_{p}\Bigl(\sum_{K\in\Sigma_p}
\frac{|\iD(K)|_p|\iF(K)|_p^s}{|\Aut(K)|}\Bigr)\Bigl(1-\frac1{p}\Bigr),
\end{array}
\end{equation}
as necessary. The second equality follows in identical fashion.
\end{proof}


\subsection{Counting cubic fields $K$ with fixed $\iD(K)$}\label{sec:small_d}
For each nonzero fundamental discriminant $d$, define a Dirichlet series
\[
\Phi_{\Sigma,d}(s) := c_{{\rm red}}+
\sum_{K\in\FF(\Sigma)_d}\frac{1}{\iF(K)^s},
\]
where $c_{{\rm red}}$ is either $1/2$ or $0$ depending on whether or not the \'etale cubic algebra $\Q\oplus\Q(\sqrt{d})$ satisfies the splitting conditions specified by $\Sigma$.
Using Kummer theory and class field theory, Cohen, Morra, and the
second author \cite{CM, CT} proved the following explicit formula for
$\Phi_{\Sigma, d}(s)$ when $P_\Sigma = 1$, i.e.~for counting all cubic fields
whose quadratic resolvent is $\mathbb{Q}(\sqrt{d})$.

\begin{theorem}[\cite{CT}, Theorem 2.5]\label{thm:cm_explicit}
For any nonzero fundamental discriminant $d$ we 
have
\begin{equation}\label{eqn_main_cubic}
c_d\Phi_d(s)=\dfrac{1}{2}M_{1, d}(s) \prod_{p\nmid 3d}
\left(1+\dfrac{1+\bigl(\frac{-3d}{p}\bigr)}{p^s}\right)
+\sum_{E\in\LL_3(d)}M_{2,E}(s)\prod_{p\nmid 3d}\left(1+\dfrac{\om_E(p)}{p^s}\right)\;,\end{equation} where:
\begin{itemize}
\item $c_d=1$ if $d = 1$ or $d < -3$, and $c_d=3$ if $d = -3$ or $d > 1$.
\item $\LL_3(d)$ is the set of cubic fields of discriminant
  $-d/3$, $-3d$, and $-27d$.  (The first case can of course only occur
  if $3 \mid d$, and the second only if $3 \nmid d$.)
\item For any cubic field $E$ and prime $p\nmid\Disc(E)$,
  we define
\begin{equation*}
\om_E(p):=\begin{cases}
2&\text{\quad if $p$ is totally split in $E$,}\\
0&\text{\quad if $p$ is partially split in $E$,}\\
-1&\text{\quad if $p$ is inert in $E$,}
\end{cases}
\end{equation*}

The $3$-Euler factors $M_{1, d}(s)$ and $M_{2,E}(s)$ are given in
the following table (taking $k = 1$):
\end{itemize}
\end{theorem}

\bigskip

\centerline{
\begin{tabular}{|c||c|c|c|}
\hline
Condition on $d$ & $M_{1, d}(s)$ & $M_{2,E}(s),\ \Disc(E) \in \{ -k^2d/3, -3k^2d \}$ & $M_{2,E}(s),\ \Disc(E) = -27k^2d$\\
\hline\hline
$3\nmid d$       & $1+2/3^{2s}$ & $1+2/3^{2s}$        & $1-1/3^{2s}$\\
\hline
$d\equiv3\pmod9$ & $1+2/3^s$    & $1+2/3^s$           & $1-1/3^s$\\
\hline
$d\equiv6\pmod9$ & $1+2/3^s+6/3^{2s}$ & $1+2/3^s+3\om_E(3)/3^{2s}$ & $1-1/3^s$\\
\hline
\end{tabular}}

\medskip
We will use this result, together with standard analytic techniques,
to count cubic fields $K$ with fixed $D(K)$ and varying $F(K)$. Such a result
was given as \cite[Proposition 6.3]{CM} and we give a version where the dependence
of the error term on $D(K)$ is specified.

We also extend these results to $P_\Sigma > 1$, counting cubic fields with specified splitting types.
The key result is where $\Sigma_p = \Q_p^3$ for each $p \mid P_\Sigma$, corresponding
to a demand that each such $p$ split completely in each cubic field being counted. 
Write $\LL_3(P_\Sigma, d)$ for the set of cubic fields whose
discriminant is $-k^2d/3$, $-3k^2d$, or $-27k^2d$, where $k$ is any
positive divisor of $P_\Sigma$. Thus the quadratic
resolvent of every field in $\LL_3(P_\Sigma, d)$ is $\Q(\sqrt{-3d})$.

\begin{theorem}\label{thm:cm_explicit_lc}
With $\Sigma_p = \Q_p^3$ for each $p \mid P_\Sigma$ and $\LL_3(P_\Sigma, d)$ defined as above, we have
\begin{equation}\label{eqn_main_cubic_lc}
c_d 3^{\omega(P_\Sigma)} \Phi_{\Sigma, d}(s)=\dfrac{1}{2}M_{1, d}(s) \prod_{p \nmid 3d P_\Sigma} \left(1+\dfrac{1+\bigl(\frac{-3d}{p}\bigr)}{p^s}\right)
+\sum_{E\in\LL_3(P_\Sigma, d)}M_{2,E}(s)\prod_{\substack{p \nmid 3d P_\Sigma}} \left(1+\dfrac{\om_E(p)}{p^s}\right)\end{equation}
provided that $\big( \frac{d}{p} \big) = 1$ for every prime $p \mid P_\Sigma$, and $\Phi_{\Sigma, d}(s) = 0$ otherwise. Here $\omega(P_\Sigma)$ denotes the number of prime divisors of $P_\Sigma$, and if $3 \mid P_\Sigma$ then the factors $M_{1, d}(s)$ and $M_{2, E}(s)$ are to be omitted.
\end{theorem}

The special cases $3 \mid P_\Sigma$ and/or $d \in \{ 1, -3 \}$ are all allowed; if $d = 1$ then $ \Phi_{\Sigma, d}(s)$ counts
cyclic cubic fields, and if $d = -3$ then the fields in $\LL_3(P_\Sigma, d)$ are cyclic. 

\begin{remark}\label{remark13}
The explicit form of $\Phi_{\Sigma,d}$ stated in Theorem \ref{thm:cm_explicit_lc} will not be used in the proofs of our main results. 
The ``average residue computation'' that is required for our proofs will be obtained indirectly from results proved using geometry-of-numbers methods.

All that is necessary for us is an asymptotic formula for the partial sums of $\Phi_{\Sigma,d}$ with 
bounds on the error; this is done in Theorem \ref{thm:cmcount} by interpreting
$\Phi_{\Sigma,d}$ as the weighted sum of incomplete Dedekind zeta functions and incomplete Artin $L$-functions, both having conductor $\ll_{\Sigma} d$.
\end{remark}

We then immediately show that $\Phi_{d,\Sigma}(s)$ can be written as such a weighted sum in the case when $\Sigma$ is an arbitrary finite collection of splitting types in the following steps:

\begin{itemize}

\item The splitting type at infinity: the sign of the discriminant of a cubic field is the same as the sign of the discriminant of its quadratic resolvent field. Hence, $\Phi_{\Sigma,d}$ will be $0$ if the prescribed splitting type at infinity is incompatible with the sign of $d$, and unchanged if it is compatible.

\item 
$(21)$ -- A prime $p$ is partially split in $K$ if and only if it is unramified in $K$ and inert in $\Q(\sqrt{d})$. Therefore, 
if $\big( \frac{d}{p} \big) = 1$ for any such prime $p$ then $\Phi_{\Sigma, d}(s) = 0$, and 
otherwise we eliminate all of the $p$-Euler factors from $\Phi_{\Sigma, d}(s)$.
\item 
$(1^2 1)$ -- A prime $p$ is partially ramified in $K$ if and only if it is ramified in $Q(\sqrt{d})$; therefore, 
$\Phi_{\Sigma, d} = 0$ if $p \nmid d$ for any such $p$, and otherwise $\Phi_{\Sigma, d}(s)$ is unchanged.
\item 
$(1^3)$ -- A prime $p$ is totally ramified in $K$ if and only if $p \mid f(K)$. Accordingly we remove the constant terms
from the $p$-Euler factors. 
\item 
The remaining primes $p$ are required to have splitting types $(111)$ or $(3)$. We handle the $(111)$ case by applying equation \eqref{eqn_main_cubic_lc} directly,
and the $(3)$ case by inclusion-exclusion. \end{itemize}

In summary, the proof of Theorem \ref{thm:cm_explicit_lc} follows from a careful reading of \cite{CM} and \cite{CT}. The proof in \cite{CM} proceeds by setting $L = \Q(\sqrt{d}, \sqrt{-3})$, and 
enumerating those cyclic cubic extensions $N_z/L$ which contain an appropriate $K$. By Kummer theory, any such extension is of the form 
$N_z = L(\sqrt[3]{\alpha})$. Writing $\alpha\mathbb{Z}_L = \mfa_0 \mfa_1^2 \mfq^3$ for squarefree integral coprime ideals $\mfa_0$ and $\mfa_1$,
the conductor $\mff(N/\Q(\sqrt{d}))$ is given (see \cite[Theorem 3.7]{CM}) by $\mfa_0 \mfa_1$ times a $3$-adic factor, and this $3$-adic factor
depends on the solubility of $x^3 - \alpha$ modulo powers of $3$. 

The splitting conditions in $K/\Q$ are equivalent to solubility in $L$ of $x^3 - \alpha$ modulo $P_\Sigma$, or modulo $3P_\Sigma$ if $3 \mid P_\Sigma$, and hence the existing machinery of \cite{CM}
is well suited to select for them. This is the reason that Theorem \ref{thm:cm_explicit_lc} has a very similar shape to Theorem \ref{thm:cm_explicit}. 

\medskip
We now proceed to explain the proof of  Theorem \ref{thm:cm_explicit_lc} in more detail.
As discussed above we may assume that $\big( \frac{d}{p} \big) = 1$ for every $p \mid P_\Sigma$, as
otherwise $\Phi_{\Sigma, d}(s) = 0$. Write $\calP = P_\Sigma$ if $3 \nmid P_\Sigma$, and 
$\calP = 3P_\Sigma$ if $3 \mid P_\Sigma$. 

\medskip
{\bf Step 1 -- Parametrization}.
Let $L = \Q(\sqrt{d}, \sqrt{-3})$ as before. 
In \cite[Proposition 2.7]{CM}, Cohen and Morra enumerate the set of cubic fields $K$ with resolvent $\Q(\sqrt{d})$; each occurs 
as the cubic subextension (unique up to isomorphism) of a field $N_z = L(\sqrt[3]{\alpha})$ with $\alpha = \alpha_0 u$, where $\alpha_0$ is determined
by the class in $I/I^3$ of the ideal $(\alpha)$, and $u$ represents an element
$\overline{u}$ of a $3$-Selmer group $S_3(L)[T]$. The notation $[T]$ indicates that $\overline{u}$ is annihilated by two particular elements of 
$\F_3[\Gal(L/\Q)]$ (one if $d = 1$ or $d = -3$).

A prime $p$ splits in such a $K$ if and only if: (1) it splits in $\Q(\sqrt{d})$, and
(2) every prime $\mfp_z$ of $L$ above $p$ splits completely in $N_z$. 
Since $L$ contains the third roots of unity, 
each such $\mfp_z$ of $L$ splits completely in $N_z$ if and only if $x^3 = \alpha$ is soluble in the completion of $L$ at $\mfp_z$. By Hensel's
lemma, if $3 \nmid \mfp_z$ this happens if and only if $x^3/\alpha \equiv 1 \pmod{ {}^\ast \mfp_z}$ is soluble in $L$. Further, if $\alpha$ is coprime 
to $3$, the primes
above $3$ split in $N_z/L$ if and only if $x^3/\alpha \equiv 1$ is soluble modulo $9$; to see this, note that if $v_3(\beta^3 - \alpha) > \frac{3}{2}$ with
$\alpha, \beta$ integral, then $v_3((\beta')^3 - \alpha) > v_3(\beta^3 - \alpha)$ with $\beta' := \beta - \frac{\beta^3 - \alpha}{3 \beta^2}$, yielding 
a sequence of $\beta_i$ converging to a solution of $x^3 = \alpha$ in each $3$-adic completion of $L$. 

\smallskip
{\bf Step 2 -- Conductors and Selmer group counting.}
In \cite[Theorem 3.7]{CM}, a formula is given for the conductor $\mff(N/\Q(\sqrt{d}))$. One writes $\alpha \Z_L = \mfa_0 \mfa_1^2 \mfq^3$ where $\mfa_0$ and $\mfa_1$ are integral
coprime squarefree ideals, has $\mfa_0 \mfa_1 = \mfa_{\alpha} \Z_L$ for an ideal $\mfa_{\alpha}$ of $\Q(\sqrt{d})$, and has that $\mff(N/\Q(\sqrt{d}))$ is the product of $\mfa_\alpha$
times a complicated $3$-adic factor, depending on the solubility of $x^3/\alpha \equiv 1 \pmod{ {}^* \mfp_z^n }$ for ideals $\mfp_z$ over $3$. They enumerate these
$3$-adic factors by inclusion-exclusion, involving a quantity
\[
f_{\alpha_0}(\mfb) = \# | \{ \overline{u} \in S_3(L)[T], \ x^3/(\alpha_0 u) \equiv 1 \pmod{ {}^* \mfb} \ \textnormal{ soluble in } L \} |,
\]
where $\mfb$ ranges over (possibly fractional) powers of $3$. This leads (\cite[Proposition 4.6]{CM}) to a formula for $\Phi_d(s)$, where $\mfb$ ranges over
a set of $3$-adic ideals $\mathcal{B}$, and $f_{\alpha_0}(\mfb)$ appears as a counting function for the number of ideals with fixed conductor. 

As discussed above, the splitting conditions in places in $S$ are equivalent to requiring that
$x^3/(\alpha_0 u) \equiv 1$ be soluble modulo other ideals. If $3 \nmid P_\Sigma$, multiply each $\mfb$ by $\calP$.
If $3 \mid P_\Sigma$, then $3$ cannot ramify in any cubic field being counted: the sum over $\mfb \in \mathcal{B}$ and all $3$-adic factors disappear from $\Phi_d(s)$. 
In place of this sum, one takes $\mfb$ equal to $\calP$. 

The computation of $f_{\alpha_0}(\mfb)$ is carried out in Section 5 of \cite{CM}, and also in Morra's thesis \cite{morra} where more detailed proofs are presented. 
One checks that, when varying $\mfb$ as above, the proofs are identical through Lemma 5.4. 

We diverge somewhat in Lemma 5.6, which computes the size of $(Z_\mfb / Z_\mfb^3)[T]$, where
$Z_{\mfb} := (\Z_L / \mfb \Z_L)^*$. By the Chinese remainder theorem, the size of this group is multiplicative in $\mfb$, so it suffices to carry out the computation for $\mfb = (9)$ or
$\mfb = (p)$ for $p$ a rational prime other than $3$. 

For $F$ equal to $L$, $\Q(\sqrt{-3D})$, or $\Q$ write $\Gamma_{F, \mfb}$
for the multiplicative group of $\Z_F / (\Z_F \cap \mfb)$ modulo cubes, so that $\Gamma_{L, \mfb} = Z_\mfb/Z_\mfb^3$ by
definition. Then for $d \neq 1, -3$ a `descent' argument similar to that
presented in the proof of \cite[Proposition 3.4]{CT} yields an isomorphism
$\Gamma_{L, \mfb}[T] \simeq \Gamma_{\Q(\sqrt{-3d}), \mfb}[1 + \tau]$, and when $d = 1$ this 
holds (tautologically) as an equality. Similarly to \cite[Lemma 5.6]{CT},
we obtain
\begin{equation}\label{eq:zb_size}
|(Z_\mfb/Z_\mfb^3)[T]| = \begin{cases}
|\Gamma_{\Q(\sqrt{-3d}), \mfb}|/|\Gamma_{\Q, \mfb}|
& \textnormal{if \ } d \neq -3 \\
|\Gamma_{\Q, \mfb}| & \textnormal{if \ } d = -3. \\
\end{cases}
\end{equation}
By direct computation, we readily check that the right side of \eqref{eq:zb_size} is $3$ in all cases.
(Recall that $\big( \frac{d}{p} \big) = 1$.)
This yields a version of \cite[Theorem 6.1]{CM}, which gives an expression for $\Phi_d(s)$ in terms
of characters of $G_\mfb := (\Cl_{\mfb}(L) / \Cl_{\mfb}(L)^3)[T]$,
with the following modifications:
\begin{itemize}
\item 
If $3 \nmid P_\Sigma$, then each ideal $\mfb \in \mathcal{B}$ is multiplied by $\calP$, and 
$|(Z_\mfb / Z_\mfb^3)[T]|$ is multiplied by $3^{\omega(P_\Sigma)}$.
\item
If $3 \mid P_\Sigma$, then the sum over $\mfb \in \mathcal{B}$ is replaced with the single choice $\mfb = \calP$; 
$(Z_\mfb / Z_\mfb^3)[T]$ has size $3^{\omega(P_\Sigma)}$. 
\end{itemize}

\smallskip
{\bf Step 3 -- Interpretation in terms of field counting.}
For $d \neq 1, -3$ the analogue of \cite[Proposition 3.4]{CT} continues to hold, yielding a `descent' isomorphism $G_\mfb \simeq H_{\mfa'}$, where
$\mfa' := \mfb \cap \Z_{\Q(\sqrt{-3d})}$, and $H_{\mfa'} :=  \big(\Cl_{\mfa'}(\Q(\sqrt{-3d})) / \Cl_{\mfa'}^3(\Q(\sqrt{-3d})) \big)[1 + \tau]$. Then, 
Proposition 4.1 of \cite{CT} uses class field theory to establish 
a bijection between pairs of nontrivial characters of $G_{\mfb}$ and cubic fields $E$. The same argument
continues to hold, with the set of cubic fields $E$ is expanded to those whose discriminant is equal to $-3d$ times the
square of any rational integer divisor of $\mfb$. The second half of the proof of Proposition 4.1 in \cite{CT} is unnecessary, as the conclusion follows more simply from Lemma 1.3 and equation (1.3) of earlier work of Nakagawa \cite{nakagawa}. 

If $d = 1$ then as before the `descent' isomorphism is replaced by an equality and we proceed identically. If $d = -3$ then we obtain
$G_\mfb \simeq \Cl_{\mfa'}(\Q) / \Cl_{\mfa'}^3(\Q)$ with $\mfa' = \mfb \cap \Z$ and a more direct
application of class field theory establishes the required bijection.

This completes the proof.

\bigskip
We turn now to the analytic consequences. Let $d\neq -3$ be a fundamental discriminant and let $\Sigma$ be a finite collection of splitting types such that $\FF(\Sigma)_d$ is non-empty. In this case $\Phi_{\Sigma,d}(s)$ is a Dirichlet series with nonnegative coefficients, and we will
see that it has a simple pole at $s = 1$ with positive residue.
Define the quantity
\begin{equation*}
N(\FF(\Sigma)_d;Z):=\#\bigl\{K\in \FF(\Sigma)_d:\iF(K)<Z\bigr\}.
\end{equation*}
Then we have the following consequence of Theorem \ref{thm:cm_explicit_lc} and its extension to arbitrary splitting conditions described after Remark \ref{remark13}.
\begin{theorem}\label{thm:cmcount}
Let $d$ and $\Sigma$ be as above.
Then we have
\begin{equation*}
N(\FF(\Sigma)_d;Z)
=\Res_{s=1}\bigl(\Phi_{\Sigma,d}(s)\bigr)\cdot Z
+O_{\epsilon}\Bigl(|\LL_3(P_{\Sigma},d)||d|^{1/6}P_\Sigma^{1/3}Z^{2/3+\epsilon}\Bigr).
\end{equation*}
\end{theorem}
\begin{proof}
As this is standard, we give a brief account. Write the left side as
\begin{equation*}
\#\bigl\{K\in \FF(\Sigma)_d:\iF(K)<Z\bigr\}
=\frac{1}{2\pi i}\int_{2 - i \infty}^{2 + i \infty} \Phi_{\Sigma, d}(s) \frac{Z^s}{s}ds.
\end{equation*}
We then write each Euler product in $\Phi_{\Sigma, d}(s)$ 
as the product of a Dedekind zeta function $\zeta_{\mathbb{Q}(\sqrt{-3d})}(s)$ or an irreducible degree $2$ Artin $L$-function, times 
a function holomorphic and bounded in any half plane $\Re(s) \geq \sigma_0 > \frac12$.

As all of the $L$-functions have conductor $\ll |d|P_\Sigma^2$, the convexity bound yields 
\begin{equation*}
\bigg|\Phi_{\Sigma, d}(s) - \frac{\Res_{s=1}\bigl(\Phi_{\Sigma,d}(s)\bigr)}{s - 1}\bigg| \ll  |\LL_3(P_{\Sigma},d)| \cdot |d|^{1/4}\big((1 + t)P_\Sigma\big)^{1/2},
\end{equation*}
uniformly in $\Re(s) = \sigma + it$ with 
$1>\sigma \geq \sigma_0$. 

Pick $T>1$ to be optimized later. We shift the contour to the left, picking up one residue at $1$, ending up with a sum of the following integrals: from $1+\epsilon\pm iT$ to
$1+\epsilon\pm i\infty$; from $1+\epsilon\pm iT$ to $\sigma_0\pm iT$; and from $\sigma_0-iT$ to $\sigma_0+iT$. The residue gives the required main term, while the sum of these integrals is
\begin{equation*}
\ll |\LL_3(P_{\Sigma},d)| \cdot \left( \frac{Z^{1+\epsilon}}{T}+Z^{\sigma_0}T^{1/2}|d|^{1/4}P_\Sigma^{1/2} \right).
\end{equation*}
The result now follows by optimizing the value of $T$ to be $Z^{1/3}/(d^{1/6}P_\Sigma^{1/3})$.
\end{proof}

\begin{remark}
The error terms in the above theorem can clearly be improved by using subconvex estimates in place of the convexity bound. However, we do not state this improvement since we have no need for it.
\end{remark}

In the case $d = -3$, the same result and proof hold, except that $\bigl(\frac{-3d}{p}\bigr) = 1$ for all $p \nmid 3d$, so that $\Phi_{\Sigma,-3}(s)$ has a double pole at $s = 1$, as opposed to a simple pole. Therefore, as explained in \cite[Corollary 7.4]{CM}, we obtain the asymptotic \eqref{eq:CM_result1} for $P_{\Sigma} = 1$, and for $P_{\Sigma} > 1$ this
generalizes to
\begin{equation}\label{eq:CM_result2}
\#\{K\in\FF(\Sigma): \iD(K) = -3, \iF(K) < Z\}
=
C_1(\Sigma) Z (\log(Z) + C_2(\Sigma) - 1) + O(P_\Sigma^{1/3 + \epsilon} Z^{2/3 + \epsilon}),
\end{equation}
where
\[
C_1(\Sigma) := C_1 \prod_{\substack{p \mid P_{\Sigma} \\ p \neq 3}} \frac{p}{3p + 6} \cdot
\prod_{\substack{p \mid P_{\Sigma} \\ p = 3}} \frac{1}{7}
\]
\[
C_2(\Sigma) := C_2 + \sum_{\substack{p \mid P_{\Sigma} \\ p \neq 3}} \frac{2 \log p}{p + 2}
+ \sum_{\substack{p \mid P_{\Sigma} \\ p = 3}} \frac{6}{7} \log 3.
\]

\section{The asymptotics of cubic fields with bounded invariants}

For a finite collection $\Sigma$ of splitting types and positive real numbers $Y$ and $Z$, define
\begin{equation*}
N(\Sigma;Y,Z):=\bigl\{K\in\FF(\Sigma):3 \neq |\iD(K)|<Y,\;\iF(K)<Z\bigr\}.
\end{equation*}
In this section, we compute asymptotics for $N(\Sigma;Y,Z)$. We will handle the `large $Y$ case'
(i.e., large $\frac{\log Y}{\log Z}$)
using Theorem \ref{thm:dh} and the
`small $Y$ case' using Theorem \ref{thm:cmcount}. Our error terms are strong enough that these ranges of $Y$ overlap, yielding 
an asymptotic estimate for all $Y$ and $Z$. 
Indeed, we obtain asymptotic formulas with different expressions for the main terms, which we may then conclude are equal. 

These results will be used in the proofs of Theorems \ref{thradical} and
\ref{thGD}(c), where $\iD$ and $\iF$ are given equal weight. For parts (a) and (b)
of Theorem \ref{thGD} we will instead use a more direct approach so as to optimize the error terms.
All of these proofs will be given 
in Section \ref{sec:ordering}. 

\smallskip
We begin with the following important uniformity estimate due to
Davenport--Heilbronn (see, e.g., \cite[Lemma 3.3]{MR2641942}).
\begin{lemma}\label{unifDH}
The number of cubic fields $K$ such that $|\Disc(K)|<X$ and $\iF(K)=f$
is bounded by $O_\epsilon(X^{1+\epsilon}/f^2)$.
\end{lemma}

The key result of this section is the following proposition:

\begin{proposition}\label{proproughcount}
Let $\Sigma$ be an finite collection of splitting types, and let $Y$ and $Z$ be positive real numbers. Then
\begin{equation*}
\begin{array}{rcl}
\displaystyle N(\Sigma;Y,Z)&=&
\displaystyle\Bigl(\sum_{\substack{f<Z}} C_1(\Sigma,f)
\Bigr)
\cdot Y
+O_{\epsilon}(Y^{5/6}Z^{2/3}+Y^{2/3+\epsilon}Z^{4/3}  P_\Sigma^{2/3} );\\[.25in]
\displaystyle N(\Sigma;Y,Z)&=&
\displaystyle \Bigl(\sum_{\substack{|d|<Y\\{\rm fund. disc} \neq -3}}
\Res_{s=1}\Phi_{\Sigma,d}(s)
\Bigr)\cdot Z+O_{\epsilon,\Sigma}\Bigl(Y^{7/6 + \epsilon}Z^{2/3+\epsilon}  P_\Sigma^{1/3 + \epsilon} \Bigr).
\end{array}
\end{equation*}
\end{proposition}
\begin{proof}
To prove the first equality, we fiber by $\iF(K)$ and apply Theorem
\ref{thm:dh} (Davenport-Heilbronn), obtaining
\begin{equation*}
\begin{array}{rcl}
\displaystyle N(\Sigma;Y,Z)&=&\displaystyle
\sum_{f<Z} N(\FF(\Sigma)^{(f)};Y)
\\[.2in]&=&\displaystyle
\Bigl(\sum_{f<Z} C_1(\Sigma,f)\Bigr)
\cdot Y+O\Bigl(\sum_{ f< Z}(f^{-1/3}Y^{5/6}+E(Y;f,\Sigma))\Bigr)
\\[.2in]&=&\displaystyle
\Bigl(\sum_{f<Z} C_1(\Sigma,f)\Bigr)
\cdot Y+O\Bigl(Y^{5/6}Z^{2/3}+Y^{2/3+\epsilon}Z^{4/3} P_\Sigma^{2/3} \Bigr),
\end{array}
\end{equation*}
as necessary.
To prove the second equality, we fiber by $\iD(K)$ and
apply Theorem \ref{thm:cmcount} (Cohen-Morra):
\begin{equation*}
\begin{array}{rcl}
\displaystyle N(\Sigma;Y,Z)&=&\displaystyle
\sum_{\substack{|d|<Y\\{\rm fund. disc} \neq -3}} N(\FF(\Sigma)_d;Z)
\\[.2in]&=&\displaystyle
\Bigl(\sum_{\substack{|d|<Y\\{\rm fund. disc} \neq -3}}
\Res_{s=1}\Phi_{\Sigma,d}(s)
\Bigr)\cdot Z+O_{\Sigma,\epsilon}\Bigl(\sum_{\substack{|d|<Y\\{\rm fund. disc} \neq -3}}
|\LL_3(P_{\Sigma},d)||d|^{1/6}Z^{2/3+\epsilon} P_\Sigma^{1/3} \Bigr)
\\[.2in]&=&\displaystyle
\Bigl(\sum_{\substack{|d|<Y\\{\rm fund. disc} \neq -3}}
\Res_{s=1}\Phi_{\Sigma,d}(s)\Bigr)\cdot Z
+O_{\Sigma,\epsilon}(Y^{7/6 + \epsilon}Z^{2/3+\epsilon} P_\Sigma^{1/3 + \epsilon}),
\end{array}
\end{equation*}
where the bound on the sum over $d$ of the sizes of $\LL_3(P_{\Sigma},d)$ follows from Lemma
\ref{unifDH}.
This concludes the proof of the proposition.
\end{proof}

Next, we estimate the leading constant in the right hand side of the
first equation of Proposition \ref{proproughcount}.
\begin{proposition}\label{propconstant}
We have
\begin{equation*}
\displaystyle\sum_{f<Z}
C_1(\Sigma,f)
=\displaystyle\frac12
\Bigl(\sum_{K\in\Sigma_\infty}\frac{1}{|\Aut(K)|}\Bigr)
\prod_{p}\Bigl(\sum_{K\in\Sigma_p}
\frac{|\iD(K)|_p|\iF(K)|_p}{|\Aut(K)|}\Bigr)\Bigl(1-\frac1{p}\Bigr)^2
\cdot Z+O_\epsilon(Z^{3/5+\epsilon});
\end{equation*}
\end{proposition}
\begin{proof}
Recall the Dirichlet series $L_1(\Sigma, s):=\sum_{f}C_1(\Sigma,f)f^{-s}$ from \S2.1. It is easy to see that
$L_1(\Sigma, s)$ is holomorphic to the right of $\Re(s)>1/2$ with a simple pole
at $s=1$. Indeed, the shape of $C_1(\Sigma,f)$ described in \eqref{eq:shapeC1} implies that $L_1(\Sigma, s)/\zeta(s)$ converges absolutely and is bounded uniformly in $\Sigma$ and $s$
to the right of $\Re(s)=\sigma$ for any $\sigma>1/2$. Pick a real number $T$ to be optimized later. Following the proof of Theorem \ref{thm:cmcount}, we have
\begin{equation}\label{eqCavg1}
\begin{array}{rcl}
  \displaystyle\sum_{f<Z}
  C_1(\Sigma,f)&=&
\displaystyle\int_{\Re(s)=2}L_1(\Sigma, s)(s)\frac{Z^s}{s} ds
\\[.2in]&=&\displaystyle
\Res_{s=1}L_1(\Sigma, s)\cdot Z+ O_\epsilon\Bigl(\frac{Z^{1+\epsilon}}{T}+
Z^{1/2+\epsilon}T^{1/4} \Bigr),
\end{array}
\end{equation}
where we use the convex bound to estimate the growth of $\zeta(s)$ (and therefore $L_1(\Sigma, s)$) on the line $\Re(s)=1/2+\epsilon$.
From the Euler product expansion of $L_1(\Sigma, s)$ derived in Proposition \ref{prop:L12Euler}, it follows that the residue of $L_1(\Sigma, s)$ at $s=1$ is given by
\begin{equation}\label{eqLres}
\Res_{s=1}L_1(\Sigma, s)=\frac12\Bigl(\sum_{K\in\Sigma_\infty}\frac{1}{|\Aut(K)|}\Bigr)\prod_{p}\Bigl(\sum_{K\in\Sigma_p}
\frac{|\iD(K)|_p|\iF(K)|_p}{|\Aut(K)|}\Bigr)\Bigl(1-\frac1{p}\Bigr)^2.
\end{equation}
The proposition follows by choosing $T=Z^{2/5}$.
\end{proof}

The above two propositions have the following consequence:
\begin{corollary}\label{corconstant}
We have
\begin{equation*}
\frac{1}{Y}\sum_{\substack{|d|<Y\\{\rm fund. disc} \neq -3}}
\Res_{s=1}\Phi_{\Sigma,d}(s)=
\frac12\prod_{p}\Bigl(\sum_{K\in\Sigma_p}
\frac{|\iD(K)|_p|\iF(K)|_p}{|\Aut(K)|}\Bigr)\Bigl(1-\frac1{p}\Bigr)^2+O_{\epsilon,\Sigma}(Y^{-1/12 + \epsilon}).
\end{equation*}
\end{corollary}
\begin{proof}
The result follows from Propositions \ref{proproughcount} and
\ref{propconstant} by setting $Z=Y^{3/4}$.
\end{proof}

In particular, the two estimates of Proposition \ref{proproughcount} are asymptotic formulas 
for $Y > Z^{1 + \epsilon}$ and $Y < Z^{2 - \epsilon}$, respectively. Since these ranges overlap, we obtain 
the following result.
\begin{theorem}\label{thdyadicmain}
We have
\begin{equation*}
N(\Sigma;Y,Z)=
\frac12\Bigl(\sum_{K\in\Sigma_\infty}\frac{1}{|\Aut(K)|}\Bigr)
\prod_{p}\Bigl(\sum_{K\in\Sigma_p}
\frac{|\iD(K)|_p|\iF(K)|_p}{|\Aut(K)|}\Bigr)\Bigl(1-\frac1{p}\Bigr)^2
  \cdot YZ+o_\Sigma(Y)Z+Yo_\Sigma(Z).
\end{equation*}
\end{theorem}
\begin{proof}
Combining the above results yields the claimed result with an error term 
\begin{equation*}
\ll_{\epsilon,\Sigma}(YZ)^\epsilon\big(YZ^{3/5}+Y^{11/12}Z+\min(Y^{2/3}Z^{4/3},Y^{7/6}Z^{2/3})\big),
\end{equation*}
which is sufficiently small.
\end{proof}

\section{Ordering cubic fields by generalized discriminants}\label{sec:ordering}
In this section we determine asymptotics for the number of cubic
fields with bounded generalized discriminant, thereby proving Theorem
\ref{thGD}. We also then determine which generalized discriminants $I$
are such that the family of cubic fields ordered by $I$ satisfy
independence of primes, thus also proving Theorem \ref{thIndPr}.

For a generalized discriminant $I = |\iD|^\alpha \iF^\beta$, after normalizing we may assume 
that one of $\alpha$ or $\beta$ equals $1$ and the other is $\geq 1$. We handle each of the three possible cases
in turn.


\begin{proposition}\label{propMT1}
For a finite collection $\Sigma$ of cubic splitting types and a
real number $\beta>1$, we have
\begin{equation*}
N_{|\iD|F^{\beta}}(\Sigma;X)=
L_1(\Sigma, \beta)\cdot X+L_2(\Sigma, 5\beta/6)X^{5/6}+O_{\epsilon,\beta}\big( (X^{\frac{2}{\beta+1} + \epsilon}+  X^{\frac{2}{3}+\epsilon})
P_\Sigma^{2/3}\big),
\end{equation*}
with $L_1(\Sigma, \beta)$ and $L_2(\Sigma, 5\beta/6)$ as given in Proposition \ref{prop:L12Euler}.
\end{proposition}
\begin{proof}
We fiber over $f\geq 1$ and write
\begin{equation*}
  N_{|\iD|F^{\beta}}(\Sigma,X)=
  \sum_{f}
  N(\FF(\Sigma)^{(f)};f^{2-\beta}X) + O(X^{\frac{1}{\beta}} \log(X)),
\end{equation*}
where
$N(\FF(\Sigma)^{(f)};f^{2-\beta}X)$ denotes the
number of cubic fields $K\in\FF(\Sigma)^{(f)}$ such that $|\Disc(K)|<f^{2-\beta}X$,
and the error term accounts for the pure cubic fields.
For any $1<Y \ll X^{1/\beta}$, by Lemma \ref{unifDH} we have
\begin{equation*}
\displaystyle\sum_{f \geq Y}
N(\FF(\Sigma)^{(f)};f^{2-\beta}X)
\ll_\epsilon \displaystyle\sum_{f \geq Y} \frac{f^{2-\beta+\epsilon}X^{1+\epsilon}}{f^2}
\ll_\epsilon \
\frac{X^{1+\epsilon}}{Y^{\beta-1-\epsilon}}.
\end{equation*}
For $f < Y$, by Theorem \ref{thm:dh} we have
\begin{equation*}
\displaystyle\sum_{f < Y}
N(\FF(\Sigma)^{(f)};f^{2-\beta}X)
=\displaystyle
\displaystyle\sum_{f < Y} \Bigl(\frac{C_1(\Sigma,f)}{f^{\beta}}X
+\frac{C_2(\Sigma,f)}{f^{5\beta/6}}X^{5/6}+O\bigl(E(X/f^\beta,f,\Sigma)\bigr)\Bigr).
\end{equation*}
The error term is bounded by
\begin{align*}
\sum_{k<\log_2 Y}\sum_{2^k\leq f< 2^{k+1}} E(X/f^\beta;f,\Sigma)
\ &\ll
\sum_{k<\log_2 Y} X^{2/3+\epsilon} (2^k)^{4/3-2\beta/3 + \epsilon}P_\Sigma^{2/3}
\\
&\ll 
\displaystyle
X^{2/3+\epsilon}P_\Sigma^{2/3} \cdot \max(Y^{4/3-2\beta/3 + \epsilon}, 1).
\end{align*}
Meanwhile, the two main terms are:
\begin{equation*}
\begin{array}{rcl}
\displaystyle\sum_{f < Y} \frac{C_1(\Sigma,f)}{f^{\beta}}
&=&\displaystyle\sum_{f\geq 1} \frac{C_1(\Sigma,f)}{f^{\beta}}+O_\beta(Y^{1-\beta})
\\[.2in]&=&\displaystyle
L_1(\Sigma, \beta)+O_\beta(Y^{1-\beta});
\\[.05in]\displaystyle
\sum_{f< Y}\frac{C_2(\Sigma,f)}{f^{5\beta/6}}
&=&\displaystyle
\sum_{f\geq 1}\frac{C_2(\Sigma,f)}{f^{5\beta/6}}+O_\beta(Y^{1-5\beta/6})
\\[.2in]&=&\displaystyle
L_2(\Sigma, 5\beta/6)+O_\beta(Y^{1-5\beta/6}).
\end{array}
\end{equation*}
Optimizing (in $X$ aspect), we
pick $Y=X^{\frac1{\beta+1}}$ and obtain the result.
\end{proof}

\begin{proposition}\label{propMT2}
For a finite collection $\Sigma$ of cubic splitting types and a
real number $\alpha>1$, we have
\begin{equation*}
N_{|\iD|^\alpha \iF}(\Sigma; X)=
\Bigl(\sum_{d\,{\rm fund.}\,{\rm disc} \neq -3}\frac{\Res_{s=1}\Phi_{\Sigma,d}(s)}{|d|^\alpha}
\Bigr)\cdot X+O_{\epsilon,\alpha}((X^{3/(2\alpha+1)+\epsilon}+X^{2/3+\epsilon})P_\Sigma^{1/3}).
\end{equation*}
\end{proposition}
\begin{proof}
We fiber over $d$ and write
\begin{equation*}
N_{|\iD|^\alpha \iF}(\Sigma, X)=
\sum_{d\,{\rm fund.}\,{\rm disc} \neq -3}N(\FF(\Sigma)_d;X/|d|^\alpha),
\end{equation*}
Pick a real number $1<Y\ll X^{1/\alpha}$ to be optimized later. For each fundamental discriminant $d$ such that $|d|\geq Y$, the condition
$|d|^\alpha f < X$ implies that $f < X/Y^\alpha$. Hence by Lemma \ref{unifDH} we have
\begin{equation*}
\begin{array}{rcl}
\displaystyle    
\sum_{\substack{d\,{\rm fund.}\,{\rm disc} \neq -3\\|d| \geq Y}}
N(\FF(\Sigma)_d;X/|d|^\alpha)
&\leq&\displaystyle
\sum_{f < X/Y^\alpha}\#\{K\in\FF(\Sigma)^{(f)}:|\Disc(K)| < X^{1/\alpha}f^{2-1/\alpha}
\}
\\[.2in]&\ll&\displaystyle
\sum_{f < X/Y^\alpha} X^{\epsilon} \cdot (X/f)^{1/\alpha}
\\[.2in]&\ll&\displaystyle
\frac{X^{1 + \epsilon}}{Y^{\alpha-1}}.
\end{array}
\end{equation*}

To estimate the main term, we use Theorem \ref{thm:cmcount} to write
\begin{equation*}
\sum_{\substack{d\,{\rm fund.}\,{\rm disc} \neq -3\\|d|<Y}}N(\FF(\Sigma)_d;X/|d|^\alpha)
=
\Bigl(\sum_{\substack{d\,{\rm fund.}\,{\rm disc} \neq -3\\|d|<Y}}
\frac{\Res_{s=1}\Phi_{\Sigma,d}(s)}{|d|^\alpha}\Bigr)\cdot X
+O(E),
\end{equation*}
where the error term $E$ is easily bounded by breaking up the sum over
$d$ into dyadic ranges and using Lemma \ref{unifDH} to
estimate the size of $\LL_3(P_{\Sigma},d)$:
\begin{equation*}
\begin{array}{rcl}
E&\ll&
\displaystyle \sum_{\substack{d\,{\rm fund.}\,{\rm disc} \neq -3\\|d|<Y}}
|\LL_3(P_{\Sigma},d)| \cdot |d|^{1/6}P_\Sigma^{1/3}(X/|d|^\alpha)^{2/3+\epsilon}
\\[.2in]&\ll&
X^{2/3+\epsilon}\max(Y^{7/6-2\alpha/3},1)P_\Sigma^{1/3}.
\end{array}
\end{equation*}
Optimizing, we pick $Y=X^{2/(2\alpha+1)}$ and obtain the required result.
\end{proof}

\begin{theorem}\label{thCond}
We have
\begin{equation}\label{eqMT}
N_{|\iD|\iF}(\Sigma;X)=
\frac12\Bigl(\sum_{K\in\Sigma_\infty}\frac{1}{|\Aut(K)|}\Bigr)
\prod_{p}\Bigl(\sum_{K\in\Sigma_p}
\frac{|\iD(K)|_p |\iF(K)|_p}{|\Aut(K)|}\Bigr)\Bigl(1-\frac1{p}\Bigr)^2
\cdot X\log X+o(X \log X).
\end{equation}
\end{theorem}
\begin{proof}
Given $\epsilon > 0$, choose $\epsilon' < \epsilon$ so that 
the interval $[1, \sqrt{X})$ may be divided exactly into 
$\frac12 (\epsilon'^{-1} + O(1)) \log X$
intervals of the form $[(1 + \epsilon')^k, (1 + \epsilon')^{k + 1})$,
and write $Y_k := (1 + \epsilon')^k$ and $Z_k := X/(1 + \epsilon')^k$.
By Theorem \ref{thdyadicmain}, we have that
\[
N(\Sigma; Y_{k + 1}, Z_k) - N(\Sigma; Y_k, Z_k) = C_1(\textnormal{DF}, \Sigma) \cdot
\epsilon' X + o(Y_k) Z_k,
\]
where $C_1(\textnormal{DF}, \Sigma)$ is the constant in \eqref{eqMT},
and the same is true with the roles of $Y$ and $Z$ reversed. Since every field
counted by $N_{|\iD|\iF}(\FF(\Sigma);X)$ is counted in one of the above rectangles,
we obtain
\[
N_{|\iD|\iF}(\FF(\Sigma);X) \leq C_1(\textnormal{DF}; \Sigma) X \log X \cdot
(1 + O(\epsilon) + \epsilon'^{-1} o_X(1)).
\]
Choosing $\epsilon \rightarrow 0$ as $X \rightarrow \infty$, we obtain the result
as an upper bound. To obtain the lower bound, proceed analogously, choosing 
$Z_k := X/(1 + \epsilon')^{k + 1}$ and subtracting the $O(X)$ fields in 
$N(\Sigma; \sqrt{X}, \sqrt{X})$ which are counted twice. 
\end{proof}

Theorems \ref{thGD} and \ref{thm:secterm} follow immediately from Propositions \ref{propMT1}
and \ref{propMT2}, and Theorem \ref{thCond}. 
We conclude
by proving Theorems \ref{thradical} and \ref{thIndPr}.

\medskip
\noindent{\bf Proof of Theorem \ref{thradical}:} For a prime $p>3$
and \'etale cubic extension $K_p$ of $\Q_p$, as noted previously we have
that $p^2\nmid \iD(K)\iF(K)$. Therefore, for a cubic field $K$, we
have ${\rm rad}(\Disc(K))=\iD(K)\iF(K)$ up to sign and bounded powers of $2$
and $3$. Let $\delta_2$ and $\delta_3$ be powers of $2$ and $3$, respectively. For $p = 2, 3$ let $S(\delta_p)$ be the set of cubic \'etale extensions $K_p$ of $\Q_p$ for which
$\iD(K_p) \iF(K_p)$ 
has $p$-adic part $\delta_p$, and let $\Sigma(\delta_2,\delta_3)$ be the finite collection of
cubic splitting types defined by $\Sigma_2=S(\delta_2)$,
$\Sigma_3=S(\delta_3)$, and $\Sigma_v=\Sigma_v^{\rm all}$ for all other
places $v$. Then we have
\begin{equation*}
\begin{array}{rcl}
\displaystyle\#\bigl\{K\in\FF(\Sigma):C(K)<X,\,\iD(K) \neq -3,\,\pm\Disc(K)>0\bigr\}&=&
\displaystyle N_{|\iD|\iF}\Bigl(\Sigma;
\frac{\delta_2\delta_3}{{\rm rad}(\delta_2\delta_3)}
\cdot X\Bigr)
\\[.2in]&=&\displaystyle
\frac{1}{2\sigma_\pm}
\CC(\delta_2)\CC(\delta_3)
\prod_{p\geq 5}\Bigl(1+\frac{2}{p}\Bigr)
  \Bigl(1-\frac{1}{p}\Bigr)^2
\cdot X\log X\\[.2in]&&+o(X\log X),
\end{array}
\end{equation*}
where $\sigma_+=6$ and $\sigma_-=2$ are the sizes of the automorphism
groups of $\R^3$ and $\R\times\C$, respectively, and
\begin{equation*}
  \CC(\delta_2)=\Bigl(1-\frac{1}{2}\Bigr)^2\sum_{K_2\in S(\delta_2)}\frac{\textnormal{rad}(\delta_2)^{-1}}{|\Aut(K_2)|};
  \quad\quad
  \CC(\delta_3)=\Bigl(1-\frac{1}{3}\Bigr)^2\sum_{K_3\in S(\delta_3)}\frac{\textnormal{rad}(\delta_3)^{-1}}{|\Aut(K_3)|}.
\end{equation*}
Summing over all $\delta_2$ and $\delta_3$, we
obtain
\begin{equation*}
\#\bigl\{K\in\FF:C(K)<X,\,\iD(K)\neq-3,\,\pm\Disc(K)>0\bigr\}=
\frac{1}{2\sigma_\pm}
\prod_{p}
\CC(p)
  \Bigl(1-\frac{1}{p}\Bigr)^2
\cdot X\log X+o(X\log X),
\end{equation*}
where for any prime $p$, the quantity $\CC(p)$ is defined to be
\begin{equation*}
\CC(p):=
\sum_{K\in\Sigma_p^\all}\frac{|{\rm rad}(|\iD(K)| \iF(K))|_p}{|\Aut(K)|}.
\end{equation*}
To compute these constants, we use the database of local fields available at \url{lmfdb.org} \cite{LMFDB}, which lists each 
quadratic or cubic ramified extension of $\Q_2$ and $\Q_3$ with its Galois group; we obtain that 
$\CC(2)=3$ and
$\CC(3)=11/3$.

Finally, to count the contribution of the pure cubic fields, observe that we have
\[
\sum_{\iD(K) = -3} C(K)^{-s} = -\frac{1}{2} \cdot 3^{-s} + \frac{3}{2} \cdot 3^{-s} \prod_{p \neq 3} \big(1 + \frac{2}{p^s}\big).
\]
by \cite[Proposition 7.3]{CM}. By an argument identical to that of 
 Theorem \ref{thm:cmcount} or
\cite[Proposition 7.4]{CM}, we have
\begin{equation*}
\#\bigl\{K\in\FF:C(K)<X,\,\iD(K)=-3,\,\pm\Disc(K)>0\bigr\}=
\frac{3}{10}
\prod_{p}
  \Bigl(1+\frac{2}{p}\Bigr)
  \Bigl(1-\frac{1}{p}\Bigr)^2
\cdot X\log X+o(X\log X),
\end{equation*}
thereby completing the proof. $\Box$

\medskip
\medskip

\noindent{\bf Proof of Theorem \ref{thIndPr}:} 
For $\alpha \leq \beta$ this follows immediately from the shape
of the leading asymptotics in parts (a) and (c) of Theorem \ref{thGD}.
It remains to
prove the result when $\alpha>\beta$, and as before we may assume
that $\beta=1$. For fixed $\epsilon>0$ let $N(\epsilon)$ be
the smallest positive integer such that the following inequality is
satisfied:
\begin{equation}\label{eqindprimes}
\epsilon\cdot\frac{\Res_{s=1}\Phi_{\Sigma^\all,5}}{5^\alpha}>
\sum_{\substack{|d|>N(\epsilon)\\{\rm fund. disc.}}}
\frac{\Res_{d=1}\Phi_{\Sigma^\all,d}}{|d|^\alpha}.
\end{equation}
Such an $N(\epsilon)$ exists for each $\epsilon$ since the sum of
$\Res_{d=1}(\Phi_{\Sigma^\all,d})/|d|^\alpha$ is convergent, as can be seen from
Corollary \ref{corconstant} for example.

For
each fundamental discriminant $d$ with $3, 5 \neq |d|\leq N(\epsilon)$, now let
$p_d\neq 7$ be a prime such that the splitting types of $p_d$ at
$\Q(\sqrt{5})$ and $\Q(\sqrt{d})$ differ. (For $d = 1$, we choose $p_1$ to be inert in $\Q(\sqrt{5})$.)
Define
$\Sigma_{p_d}^{(\epsilon)}$ to be the set of all \'etale cubic
extensions of $\Q_{p_d}$ whose quadratic resolvents
are equal to $\Q_{p_d}(\sqrt{5})\oplus\Q_{p_d}$.
We then define the collection $\Sigma^{(\epsilon)}$ by taking 
$\Sigma_{p_d}^{(\epsilon)}$ as above, and choosing 
$\Sigma_{p}^{(\epsilon)}=\Sigma_p^\all$ for $p$ not equal to any of the
$p_d$.

Then \eqref{eqindprimes} holds with $\Sigma^{\all}$ replaced by $\Sigma^{\epsilon}$ and
$|d| > N(\epsilon)$ replaced by $d \neq -3, 5$, as
the newly imposed splitting conditions do not exclude any of the fields counted on the left,
nor do they include any of the fields added to the right.
Since $K\otimes\Q_7\not\cong\Q_7^3$ for any $K$ with resolvent $\Q(\sqrt{5})$, this implies that
\[
0 < \lim_{X \rightarrow \infty} 
\frac{\#\{K\in\FF(\Sigma^{(\epsilon)}):\iD(K) \neq -3,\, \iD(K)^\alpha\iF(K)<X,\, K\otimes\Q_7\cong\Q_7^3\} }
{\#\{K\in\FF(\Sigma^{(\epsilon)}):\iD(K) \neq -3,\, \iD(K)^\alpha\iF(K)<X \} }
< \frac{\epsilon}{1 + \epsilon} < \epsilon.
\]

In particular, since $7$ does not split in $\Q(\sqrt{5})$, the probability of the prime $7$ splitting
completely in $\FF(\Sigma^{(\epsilon)})$ goes to $0$ as $\epsilon$ tends to $0$.
Moreover, the probability that $7$ splits completely in $\FF(\Sigma^\all)$ is positive, since $7$ splits completely a positive proportion of the time in cubic fields with resolvent, say, $\Q(\sqrt{-19})$.
The result now follows from the fact that $\Sigma^{(\epsilon)}_7$ is constant for all
$\epsilon$.
$\Box$

\section{Numerical data}

As a double check on our work we numerically verified Theorem \ref{thm:cm_explicit_lc} (the explicit
Dirichlet series counting cubic fields with local conditions, by quadratic resolvent), and the
$\alpha = \beta = 1$ case of Theorem \ref{thGD} (counting cubic fields with $|\iD|\iF < X$).
Our code can be readily modified to cover additional cases of Theorem \ref{thGD}.

We used the PARI/GP programming language \cite{pari},
and our source code and data may be downloaded from {\upshape \url{thornef.github.io}}: a program \href{https://thornef.github.io/cm-test.gp}{\texttt{cm-test.gp}}
to compute instances of Theorem \ref{thm:cm_explicit_lc}, and compare against known data
when possible; a program \href{https://thornef.github.io/cubic-count.gp}{\texttt{cubic-count.gp}} to
generate the data below; and lists of cubic fields (\href{https://thornef.github.io/rcf-1500k.gp}{\texttt{rcf-1500k.gp}} and
\href{https://thornef.github.io/icf-1000k.gp}{\texttt{icf-1000k.gp}})
obtained from {\upshape \url{lmfdb.org}} \cite{LMFDB}.

For counting cubic fields with $|\iD|\iF < X$, the following table presents a comparison
(for relatively small $X$) of the asymptotics proved in Theorem \ref{thGD}(c) with the data:

\begin{center}
\begin{tabular}{c | c | c}
$X$ & Theorem \ref{thGD}(c) & Actual Data \\ \hline
100 & 50 & 38\\
1000 & 748 & 629 \\
10000 & 9977 & 9181 \\
20000 & 21456 & 20044 \\
30000 & 33502 & 31427 \\ 
\end{tabular}
\end{center}

We note the apparent presence of one or more negative lower order terms. There are at least
three possible explanations for the discrepancy between the data and the asymptotics:
\begin{itemize}
\item
the negative secondary term in the Davenport-Heilbronn theorem \eqref{eq:dh}; 
\item the exclusion of $\iD = -3$ from our counts, which is not `visible' in the main term of 
Theorem \ref{thGD}(c);
\item the natural tendency for asymptotics with logarithmic terms to have lower order terms
without the logarithms, e.g. the divisor sum estimate $\sum_{n < X} d(n) = X \log X + (2\gamma - 1)X + O(\sqrt{X})$. 
\end{itemize}

We leave a more detailed analysis for followup work. 
\subsection*{Acknowledgments}

It is a pleasure to thank Manjul Bhargava, Robert Lemke Oliver, Tim Santens, and Takashi Taniguchi for helpful conversations. We are extremely grateful to the anonymous referee for a very thorough reading and many useful suggestions. 

AS is supported by an
NSERC Discovery Grant and a Sloan Research Fellowship. FT 
was partially supported by the National Science Foundation under Grant
No. DMS-2101874, by the National Security Agency under Grant
H98230-16-1-0051, and by grants from the Simons Foundation (Nos.~563234 and~586594).

Finally, this project was begun at the Palmetto Number Theory Series in Columbia, SC in
December 2016, hosted by Matthew Boylan, Michael Filaseta, and FT. We would like to thank
Boylan and Filaseta for hosting, and the NSF (Grant No. DMS-1601239) and NSA 
(Grant No. H98230-14-1-0151) for funding the conference. 

\bibliographystyle{abbrv}
\bibliography{references}

\end{document}